\documentclass[11pt]{article}
\usepackage{amsmath,amssymb,amsthm,amsfonts,mathrsfs,mathtools,physics}

\usepackage[nottoc]{tocbibind}
\usepackage{dsfont} 
\usepackage[utf8]{inputenc}
\usepackage[T1]{fontenc}
\usepackage[english]{babel}
\usepackage{palatino}
\usepackage{comment}
\usepackage{fullpage}
\usepackage{enumerate}
\usepackage[toc]{appendix}
\usepackage{esint}
\usepackage{appendix}
\usepackage{url}
\usepackage{graphics,graphicx}
\usepackage{ifthen}
\usepackage[normalem]{ulem} 
\usepackage{caption}
\usepackage{tikz}

\usepackage[backref=page]{hyperref}
 \hypersetup{
 	plainpages=false,
 	unicode=false,          
 	pdftoolbar=true,        
 	pdfmenubar=true,        
 	pdffitwindow=true,      
 	pdftitle={My title},    
 	pdfauthor={},     
 	pdfsubject={},   
 	pdfcreator={},   
 	pdfproducer={}, 
 	pdfkeywords={}, 
 	pdfnewwindow=true,      
 	colorlinks=true,       
 	linkcolor=blue,          
 	citecolor=blue,        
 	filecolor=blue,      
 	urlcolor=blue,          
 	urlbordercolor=0 1 1
 }
\usepackage{xcolor}
\definecolor{Blue}{rgb}{0.,0.,1}

\usepackage{setspace}
\usepackage[pass]{geometry}
\newcommand{\lm}{\lim\limits_{\e\rightarrow 0}}

\newcommand{\R}{\mathbb R}
\newcommand{\N}{\mathbb N}

\newcommand{\cF}{\mathcal{F}}

\newcommand{\cL}{\mathcal{L}}

\newcommand{\G}{{\mathbb G}_{d,n}}
\newcommand{\V}{\|V\|}

\newtheorem{theo}{Theorem}[section]
\newtheorem*{theo*}{Theorem}
\newtheorem{prop}[theo]{Proposition}
\newtheorem{lemma}[theo]{Lemma}
\newtheorem{dfn}[theo]{Definition}
\newtheorem*{dfn*}{Definition}
\newtheorem{cor}[theo]{Corollary}

\newtheoremstyle{rmdotless}{}{}{\upshape}{}{\bfseries}{.}{0.5em}{}
\theoremstyle{rmdotless}
\newtheorem{remk}[theo]{Remark}

\DeclareMathOperator*{\supp}{spt}

\renewcommand{\phi}{\varphi}
\renewcommand{\epsilon}{\varepsilon}
\newcommand{\e}{\epsilon}
\renewcommand{\G}{G_{d,n}}

\def\L{\mathcal{L}} 

\def\ds{\displaystyle}

\title{}
\date{}
\begin{document}

\begin{center}
\rule{\textwidth}{0.5pt}
\vspace{0.2cm}
{\Large \textbf{Approximations of the mean curvature, and the Buet-Rumpf approximate mean curvature flow}}
\vspace{0.5cm}
\rule{\textwidth}{0.5pt}
\end{center}
\begin{center}
 {\bf Abdelmouksit Sagueni}\\
Institut Camille Jordan\\
{\it sagueni@math.univ-lyon1.fr}
\end{center}

\vspace{0.5cm}

\begin{abstract}
The aim of this paper is to generalize the work of B. Buet and M. Rumpf on some definition of the approximate mean curvature vector for varifolds, and its associated mean curvature motions for points clouds. We propose a generalization of the definition of the approximate mean curvature vector in two terms: in terms of linear operators and in terms of regularity of the varifold. We then extend the results to the approximate second fundamental form. Finally, we prove some additional comparison principles satisfied by the motion of points cloud by mean curvature (in the discrete and the continuous cases).

\vspace{0.5cm}
\end{abstract}

\tableofcontents

\newpage

\section{Introduction}
The mean curvature flow is an emblematic geometric flow, defined for smooth surfaces, allowing to reduce the area in the fastest possible way. One of its remarkable applications, in the discrete setting, is image denoising. Several approximations of the mean curvature and the mean curvature flow were proposed (ex: \cite{yang2007direct,dz1,feng,merigot2010voronoi,blm1,br}).\\

In this paper we generalize some results of \cite{br} on the approximation of the mean curvature vector and on the comparison principles satisfied by the motion of points cloud by approximate mean curvature. We work in the varifold setting, a varifold is a Radon measure on $\R^n \times \G$, for more context, see \cite{menne2017concept} for a simple and concise introduction, and \cite{simon, all, blm1} for more details.\\

We recall the definition of the approximate mean curvature w.r.t linear operators introduced  in \cite{br} after \cite{blm1}. For two  non-negative smooth real valued functions $\rho$ and $\xi$ defined on $\R^+$ and supported on $[0,1]$, a scale of smoothing $\e \in(0,1)$, we define for any varifold $V\in V_d(\R^n)$:
\begin{equation}\label{comp:def_mc}
H_{\rho,\xi,\e}^{\Pi,V}(x):= -\frac{C_{\xi}}{C_{\rho}} \frac{\e^{-1} \displaystyle\int_{\R^n} \rho'\left( \frac{|y-x|}{\e}\right) \frac{\Pi_y(y-x)}{|y-x|}d\V(y)}{ \xi_{\e}\ast\V(x)}
\end{equation}
whenever $\xi_{\e}\ast\V(x)>0$, where $C_{\xi}$ and $ C_{\rho}$ are normalization constants, and $\Pi_y$ is a linear operator for any $y\in \supp \V$. Denote by $H(x,V)$ the mean curvature vector of $V$,  It is known from \cite[Theorem 4.3]{blm1} that in case $V$ is the  varifold associated to a $C^2$ submanifold (more generally, $V$ is rectifiable) $\lm H_{\rho,\xi,\e}^{T_{\cdot}V,V}(x) = H(x,V)$.
\vspace{0.2cm}

\noindent {\bf Question:} for what choice of $\Pi$, do we have $\lm H_{\rho,\xi,\e}^{\Pi,V}(x) = H(x,V)$? 
\vspace{0.2cm}

Our main results on the approximation of the mean curvature are Proposition \ref{propbr3.3_gen1} and Proposition \ref{propbr3.3_gen2}. They consist of generalizations of \cite[Prop. 3.3]{br} in terms of the choice of the linear operator and the regularity of the varifold. Similarly to the original result of \cite[Proposition 3.3]{br} we do not provide an exact convergence rate, as it would require a finer study of the geometry of the submanifold (in the $C^2$ case) which is far from the subject of this work.

\noindent In the same spirit we generalize in section \ref{sec:sff_app} the definition of the approximate second fundamental form defined in \cite{blm2} (after \cite{hut1}) in terms of linear operators and regularity.
%

%
\vspace{0.2cm}

In the second part, we will investigate the comparison principles satisfied by the the motion of points cloud varifolds introduced in \cite{br}. In \cite{br}, Buet and Rumpf introduced several continuous and discrete schemes expressing the evolution of a points cloud varifold by approximate mean curvature, each corresponds to  a certain choice of the linear operator and a certain definition of the mass and the projectors. The authors proved the sphere barrier to internal varifolds for the special case where $\Pi =2 {\rm Id}$ and the kernels $\rho$ and $\xi$ satisfy the natural kernel pair property, namely:
\begin{equation}\label{nkpcondition}
 -s\rho'(s)=n\xi(s) \quad s \in \R^+.
\end{equation}
\noindent The starting point of their work on the comparison principles is the following theorem by Brakke.

\begin{theo}(Sphere comparison principles, \cite[Chapter 3]{brakke} )\\
\noindent Let $V_0 \in V_d(\R^n)$, $(V(t))_{t\geq0}$ a Brakke flow starting from $V_0$. The following principles holds
\begin{enumerate}
\item The sphere barrier to external varifolds principle:
\begin{equation*}
 \supp \|V_0 \| \subset B(a,R) \,\, \implies \,\, \supp \|V(t) \| \subset B(a,\sqrt{R^2-2dt}), \quad 
 \,\, \forall (a,R) \in \R^n \times \R^+ \, \text{and} \, t \in [0,R^2/2d].
\end{equation*}

\item The sphere barrier to internal varifolds principle:
\begin{equation*}
\supp \|V_0 \| \subset B(a,R)^c \,\, \implies \,\, \supp \|V(t) \| \subset B(a,\sqrt{R^2-2dt})^c 
\quad \\, \forall (a,R) \in \R^n \times \R^+ \, \text{and} \, t \in [0,R^2/2d].
\end{equation*}
\end{enumerate}
\end{theo}
\noindent We prove in section \ref{sec:comp_ppl}, for the continuous and discrete motions of points clouds varifolds defined by Buet and Rumpf in \cite{br}, under the same assumptions of the authors, that:
\begin{itemize}
 \item The continuous motion satisfies the sphere barrier to external varifolds principle.
 \item Explicit schemes of the discrete motion satisfy the sphere barrier to internal varifolds principle.
 \item Implicit schemes of the discrete motion satisfy a "weak'' sphere barrier to external varifolds principle (weak in the sense that the time step depends on the distances between the points and the center of the sphere).
\end{itemize}

\section{Approximations of the mean curvature}
%
%

%

In this section we propose several definitions of the approximate mean curvature vector in light of \cite[Proposition 3.3]{br}. We prove the convergence of these approximations first in the $C^2$ case (Proposition \ref{propbr3.3_gen1}) and then for unit density varifolds with  certain integrability condition on the mean curvature (Proposition \ref{propbr3.3_gen2}).
 
\noindent For simplicity, we assume that $0$ is a generic point $\in \supp \| V\|$ for every varifold $V$ we deal with in this section. The approximate mean curvature at $0$ is given by:
\begin{equation*}
 H_{\rho,\xi,\epsilon}^{\Pi}(0,V)=-
\frac{C_{\xi}}{C_{\rho}\e} 
\frac{ \ds\int_{\R^{n}}\rho'\left(\frac{\vert y \vert}{\epsilon}\right)\Pi_y\left(\frac{y}{\vert y \vert}\right) d \V(y)} { \e^{d} \left( \ds\xi_{\e} \ast \V \right)(0) }
\end{equation*}
 where
\begin{equation}\label{def:cxi_crho}
C_{\rho}=d\omega_d \int_{0}^{1}\rho(t)t^{d-1}dt \quad  \text{and} \quad C_{\xi}=d\omega_d \int_{0}^{1}\xi(t)t^{d-1}dt 
\end{equation}
In the current section, we are concerned only with rectifiable varifolds (integral varifolds to be more precise), this allows the following simplifications of the approximate mean curvature vector formula. If $V$ is rectifiable, we can infer that: 
\begin{equation*}
  \lm \xi_{\e} \ast \V (x) = \lim\limits_{\e \rightarrow 0} \e^{-d} \int_{\R^{n}} \xi \left( \frac{\vert  y \vert}{\epsilon} \right) d\V(y) = \int_{T_0V} \xi(z) dz= C_{\xi}
\end{equation*}
(where $T_0V$ is the approximate tangent space of $V$ at $0$) and if we assume for simplicity that $C_{\rho} \equiv 1$, and denote  $H_{\rho,\xi,\epsilon}^{\Pi} := H_{\e}^{\Pi} $ we get the formula:
\begin{equation}\label{eq:simpleformula}
 \lm H_{\e}^{\Pi}(0,V) = -\lm {\e}^{-d-1}\int_{\R^{n}}\rho'\left(\frac{\vert y \vert}{\epsilon}\right)\Pi_y\left(\frac{y}{\vert y \vert}\right) d \V(y),
\end{equation}
this formula will be heavily used in the sequel for its simplicity. From now on, we denote $H_{\e}^{\Pi} := H_{\e}^{\Pi}(0,V)$ when there is no room for confusion.\\

\noindent{\bf Notation:} denote for simplicity $T=T_0V$ and $S=T_yV, \, y \in \supp \|V\|$.\\

The following result of \cite{br} is the starting point of our works.
\begin{prop}[Proposition 3.3, \cite{br}] \label{propbr3.3}  For a $C^2$ submanifold $M$ of dimension $d$ in $\R^n$ 
and $$ \Pi \in  \lbrace S , -2S^{\bot}  , 2{\rm Id}, T^{\bot}\circ S, -2T^{\bot}\circ S^{\bot}, 2T^{\bot} \rbrace$$
one has 
$$ \lim\limits_{\e \rightarrow 0} H_{\epsilon}^{\Pi} = H (0,M).$$
\end{prop}

We state its generalization in terms of linear operators:
\begin{prop}[Generalization in terms of linear operators]\label{propbr3.3_gen1}
For a $C^2$ submanifold $M$ of dimension $d$ in $\R^n$ and $\Pi$ in the set: 
$$ \lbrace 2S^{\bot} \circ T^{\bot}, S\circ T, -S^{\bot} \circ T \rbrace $$
we have $ \lim\limits_{\e \rightarrow 0} H_{\epsilon}^{\Pi} = H(0,M).$
Moreover, $\lm H_{\e}^{\Pi}=0$ for $\Pi$ in the set: 
\begin{equation*}
\lbrace T, T\circ S, S\circ T^{\bot}, T\circ S^{\bot} \rbrace.
\end{equation*}
\end{prop}
\begin{remk}(General approximation of the mean curvature)
In fact, we are able to determine $\lm H_{\e}^{\Pi}$ for any $\Pi$ in the algebra generated by $2 {\rm Id}$, $T$, $S$,$T^{\bot}$ and $S^{\bot}$, and it is always parallel to $H(0,M)$.
\end{remk}
We first start by stating some technical lemmas that will be useful for the proofs.
\begin{lemma}\label{lemma:ellipse} Let $V$ be  a $d$-integral varifold in $\R^{n}$, assume that $H(0,V) \in\L^{1}(\V)$ with $(\delta V)_s \equiv 0$, then 
\begin{equation*}
  \lm H_{\e}^{\Pi}
 =-\lm {\e}^{-d-1}\int_{\abs{y}\leq \e}\rho'\left(\frac{\vert T(y) \vert}{\epsilon}\right)\Pi_y\left(\frac{y}{\vert y \vert}\right) \, d \V(y).
\end{equation*}
for any linear operator $\Pi$.
\end{lemma}
\begin{proof} The main idea is to use the height-excess decay estimate on formula \eqref{eq:simpleformula} to get rid of the orthogonal component.

\noindent Let  $\Pi$ be a linear operator. As $\rho$ is supported on $[0,1]$, we have
\begin{equation}\label{eq:ellipse}
\begin{split}
&\abs{ \int_{\R^{n}}\rho'\left(\frac{\vert y \vert }{\epsilon}\right) \Pi_y\left(\frac{ y }{\vert y \vert}\right) d \V -  \int_{\abs{y}\leq \e} \rho'\left(\frac{\vert T(y) \vert }{\epsilon}\right)\Pi_y \left(\frac{y}{\vert y \vert}\right) d \V } 
\\& = 
\abs{ \int_{\abs{y}\leq \e}\rho'\left(\frac{\vert y \vert }{\epsilon}\right) \Pi_y\left(\frac{ y }{\vert y \vert}\right) d \V -  \int_{\abs{y}\leq \e} \rho'\left(\frac{\vert T(y) \vert }{\epsilon}\right)\Pi_y \left(\frac{y}{\vert y \vert}\right) d \V }
\\& \leq
\int_{\vert y \vert \leq \e } \abs{ \rho' \left(\frac{\vert y \vert}{\epsilon}\right)- \rho'\left(\frac{\vert T(y) \vert}{\epsilon}\right) } \abs{\Pi_y\left(\frac{y}{\vert y \vert}\right)  }d \V
\\&
\leq \| \Pi\|
\int_{\vert y \vert \leq \e } \abs{ \rho' \left(\frac{\vert y \vert}{\epsilon}\right)- \rho'\left(\frac{\vert T(y) \vert}{\epsilon}\right) } d \V
\\& \leq \| \Pi\| \e^{-1} \int_{\vert y \vert \leq \e} \rho''(|z|/\e) \big\vert \vert y \vert - \vert T(y) \vert \big\vert d\V, \quad \text{with} \, |T(y)| \leq |z| \leq |y|
\\& = \| \Pi\|
\e^{-1} \int_{\vert y \vert \leq \e} \rho''(|z|/\e)  \frac{ \big\vert \vert y \vert^2 - \vert T(y) \vert^2 \big\vert}{ \vert y \vert + \vert T(y) \vert} d\V 
\\&
\leq \| \Pi\| \|\rho'''\|_{\infty} \e^{-2} \int_{\vert y \vert \leq \e} \frac{|z|}{\vert y \vert + \vert T(y) \vert} \abs{ \abs{y}^2 - \abs{T(y)}^2}d\V \quad \text{ as $\rho''(0)=0$ } 
\\& = \| \Pi\| \|\rho'''\|_{\infty} \e^{-2}  \int_{\vert y \vert \leq \e} \big\vert T^{\bot}(y) \big\vert^2 d\V \\& =o(\e^{d+1}) \quad  \text{where we used the height-excess decay (Theorem \ref{thm:tilttex-heightex})}.
\end{split}
\end{equation}
We know from \eqref{eq:simpleformula} that
\begin{equation*}
 -\lm H_{\e}^{\Pi} = \lm {\e}^{-d-1}\int_{\R^{n}}\rho'\left(\frac{\vert y\vert}{\epsilon}\right)\Pi_y\left(\frac{y}{\vert y \vert}\right) d \V(y)
\end{equation*}
this finishes the proof of Lemma \ref{lemma:ellipse}.
\end{proof}
\begin{remk}[Convergence rate]
In case $H(\cdot,V) \in \cL^p(\V), p\geq 2$ and $(\delta V)_s \equiv 0$  \cite[Corollary 3.7]{menne08} infers that the height-excess decay is of the order $O(\e^{d+2})$ (instead of $o(\e^{d+1})),$ therefore we have a better estimate in this case.
\end{remk}
 \begin{remk}[Robustness to orthogonal noise] Lemma ~\ref{lemma:ellipse} allows to prove the robustness of the approximation of the mean curvature to orthogonal noise. In fact, the proof can be straightforwardly adapted to prove the following:
 \begin{equation*}
  -\lm {\e}^{-d-1}\int_{\abs{y}\leq \e}\rho'\left(\frac{\vert T(y) + \beta(y) T^{\perp}(y) \vert}{\epsilon}\right)\Pi_y\left(\frac{y}{\vert y \vert}\right) \, d \V(y) = \lm H_{\e}^{\Pi} 
\end{equation*}
for any bounded function $\beta$. The contribution of points far away from the tangent space ($T$) gets smaller as $\beta$ gets larger, in practice, this allows to  eradicate orthogonal noise when approximating the mean curvature.
 \end{remk} 

\begin{lemma}\label{lemma1_prop3.3br}
Let $V\in V_d(\R^n)$ be an integral unit density varifold, assume that $H(\cdot,V) \in \cL^p(\V), p > 2d$ and $(\delta V)_s \equiv 0$. Let
\begin{equation*}
 \begin{array}{ll}
 F : & T \rightarrow \R^n\\
  & z \mapsto (z,f(z))
\end{array}
\end{equation*}
be a local parametrization of $\supp \|V\|$ near $0$, which exists thanks to Theorem \ref{thm:allard_reg}. Then
\begin{equation}\label{S_expansion}
S=T + \begin{pmatrix}
  0 & Df^t \\
  Df & 0
 \end{pmatrix} + O(\|Df\|^2)I_n,
\end{equation}
and 
\begin{equation}\label{eq:pre_tilt_ecess}
 \abs{JF-1} \lesssim  \|Df\|^2 \lesssim \| S - T \|^2,
\end{equation}
where $JF$ is the Jacobian of the map $F$. In addition,
\begin{equation}\label{eq:intermediate1_prop3.3br}
 \lm H_{\e}^{\Pi} = - \lm {\e}^{-d-1}\int_{A_{\e}}\rho'\left(\frac{\vert z \vert}{\epsilon}\right)\Pi\left(\frac{(z,f(z))}{\vert (z,f(z)) \vert}\right) d \L^d(z),
\end{equation}
where $A_{\e} = \{ z \in T, |(z,f(z))|\leq \e \}$ or $\{ z \in T, |z| \leq \e \}$.
\end{lemma}
%
\begin{proof}
For simplicity we will omit the variable $z(\in T)$ when there is no risk of confusion. We recall that $F$  is a parametrization above the tangent space $T$ and that $f(0)=0$, $Df(0)=0$.

\noindent We start with the proof of  \eqref{S_expansion}. We first recall the expression of the projector $S$ in the matricial form, we have

\begin{equation*}
 S = DF(DF^tDF)^{-1}DF^t \in \mathcal{M}_n.
\end{equation*}

\noindent Let $M:=(DF^tDF)^{-1} \in \mathcal{M}_d $, we have
\begin{equation*}
 M^{-1}= DF^tDF = I_d + Df^tDf
\end{equation*}
thus $M^{-1} = I_d + O(\|Df\|^2)I_d$. Therefore $M = I_d + O(\|Df\|^2)M$, which yields $\|M\|=O(1)$ and 
$$M= I_d + O(\|Df\|^2)I_d.$$ We carry on with the expansion of $S$, from what is obtained previously we have, as $DF$ is bounded near $0$,
\begin{equation*}
 S= DF(DF^tDF)^{-1}DF^t = DF DF^t + O(\|Df\|^2)I_n = 
 \begin{pmatrix}
 I_d & 0 \\
 0 & 0
\end{pmatrix}
+ \begin{pmatrix}
  0 & Df^t \\
  Df & 0
 \end{pmatrix} + O(\|Df\|^2)I_n,
\end{equation*}
 this finishes the proof of \eqref{S_expansion}.

\noindent We now prove \eqref{eq:pre_tilt_ecess}.

\noindent Indeed, we write down the  Taylor expansion for the Jacobian:
\begin{equation}
 JF^2 = \det \left((DF)^{t}(DF)\right) = \det \left( 
 \begin{pmatrix}
 I_d&\\
 Df
\end{pmatrix}^t
\begin{pmatrix}
 I_d&\\
 Df
\end{pmatrix}
 \right)= 1 + O(\norm{Df}^2).
 \end{equation}
This implies that $ |JF^2 -1 | = O(\|Df\|^2)$, hence
\begin{equation}
 |JF-1| =\frac{|JF^2 -1 |}{JF+1} \leq |JF^2 -1 | = O(\|Df\|^2).
\end{equation}
and this finishes the first part of the proof of inequality \eqref{eq:pre_tilt_ecess}. For the second part, we use \eqref{S_expansion} to infer that 
\begin{equation}
 \| S - T \| = \norm{ \begin{pmatrix}
  0 & Df^t \\
  Df & 0
 \end{pmatrix} }+ o(\|Df\|)I_n.
\end{equation}
One can check easily that $ \norm{\begin{pmatrix}
  0 & Df^t \\
  Df & 0
 \end{pmatrix} }= \|Df\|$
this yields
\begin{equation}
\Big| \| S - T \| -  \|Df\| \Big| \leq \norm{ (S-T) - \begin{pmatrix}
  0 & Df^t \\
  Df & 0
 \end{pmatrix} }  = o(\|Df\|)
 \end{equation}
as a result
\begin{equation}
 \| Df \| \lesssim \|S-T\| \lesssim \|Df\|  
\end{equation}
and this finishes the proof \eqref{eq:pre_tilt_ecess}.

\noindent Finally, we provide the proof of \eqref{eq:intermediate1_prop3.3br}. The main idea is to  use the tilt-excess decay estimate (Theorem \ref{thm:tilttex-heightex}) to get rid of the Jacobian after applying the area formula.\\
By Lemma ~\ref{lemma:ellipse}, we have : 
$$ \lm H_{\e}^{\Pi} = 
 - \lm {\e}^{-d-1}\int_{\abs{y}\leq \e}\rho'\left(\frac{\vert T(y) \vert}{\epsilon}\right)\Pi\left(\frac{y}{\vert y \vert}\right) d \V.$$
Denote $\cF_{\e}:= \{ z \in T, |(z,f(z))|\leq \e \}$. The area formula implies
\begin{equation}
\int_{\abs{y}\leq \e}\rho'\left(\frac{\vert T(y) \vert}{\epsilon}\right)\Pi\left(\frac{y}{\vert y \vert}\right) d \V = 
 \int_{\cF_{\e}}\rho'\left(\frac{\vert z \vert}{\epsilon}\right)\Pi\left(\frac{(z,f(z)}{\vert (z,f(z)) \vert}\right) JF(z) \, d \L^d(z).
\end{equation}
We now show that
\begin{equation}\label{lemma1_prop3.3br_prf1}
  \abs{\int_{\cF_{\e}}\rho'\left(\frac{\vert z \vert}{\epsilon}\right)\Pi\left(\frac{(z,f(z)}{\vert (z,f(z)) \vert}\right) JF(z) \, d \L^d(z)-
 \int_{\cF_{\e}}\rho'\left(\frac{\vert z \vert}{\epsilon}\right)\Pi\left(\frac{(z,f(z))}{\vert (z,f(z)) \vert}\right) d \L^d(z) } = o(\e^{d+1}).
\end{equation}
Indeed,
\begin{equation*}
 \begin{split}
    &\abs{ \int_{\cF_{\e}}\rho'\left(\frac{\vert z \vert}{\epsilon}\right)\Pi\left(\frac{(z,f(z))}{\vert (z,f(z)) \vert}\right) JF(z) \, d \L^d(z) -
    \int_{\cF_{\e}}\rho'\left(\frac{\vert z \vert}{\epsilon}\right)\Pi\left(\frac{(z,f(z))}{\vert (z,f(z)) \vert}\right) \, d \L^d(z)
    } \\&  \leq\|\rho'\|_{\infty} \|\Pi\|  \int_{\cF_{\e}} \abs{JF(z)-1} \, d\L^d(z)
    \lesssim   \int_{\cF_{\e}} \norm{S-T}^2 \,d\L^d(z).
 \end{split}
\end{equation*}
Using $1 \lesssim JF(z)$ and applying the area formula we get: (restoring the variables for clarity)
\begin{equation*}
 \begin{split}
\int_{\cF_{\e}} \norm{S(F(z))-T}^2 d\L^d(z)
& \lesssim \int_{\cF_{\e}} \norm{S(F(z))-T}^2 JF(z) d\L^d(z)  
\\& =\int\limits_{\abs{y}\leq \e} \norm{S(y)-T}^2 d\V(y) = o(\e^{d+1}) \quad \text{by the tilt-excess decay estimate}.
 \end{split}
\end{equation*}
 this proves the first part of  \eqref{eq:intermediate1_prop3.3br}.\\
\noindent For $\delta \geq 0$ we  denote : $D_{\delta} := \{ z \in T , |z|\leq \delta \}$. We now prove that
$$\cL^d \left( D_{\e} \setminus \cF_{\e} \right)=o(\e^{d+1}).$$
\noindent Let $\alpha =1 - \frac{d}{p}$; indeed, as $f(0)=0$, $Df(0)=0$ and $f\in C^{1,\alpha}$, we can assert that
\begin{equation*}
 \sup\limits_{|z| \leq  \e} |f(z)| \leq c \,\e^{1+\alpha}
\end{equation*}
for some constant $c$ not depending on $\e$. For $\e$ small enough such that $ c^2\e^{2\alpha} < 1$ , let $\beta = \e (1-c^2\e^{2\alpha})^{\frac12}$, $z \in D_{\beta}$ implies that
\begin{equation}
 |z|^2 \leq \e^{2}(1-c^2 \e^{2\alpha}) =  \e^2 - c^2\e^{2(1+\alpha)} \leq \e^2 - |f(z)|^2,
\end{equation}
hence $z\in \mathcal{F}_{\e}$; and 
\begin{equation}
 D_{\beta} \subset \mathcal{F}_{\e} \subset D_{\e}.
\end{equation}
Consequently 
\begin{equation}
 \cL^{d} \left(D_{\e} \setminus  \cF_{\e} \right) \leq \cL^{d}  \left( D_{\e} \setminus D_{\beta} \right) \leq \omega_d \e^{d} \left( 1 - (1-c^2 \e^{2\alpha})^\frac{d}{2} \right).
\end{equation}
By mean value theorem applied to the function $ x \to x^{d/2}$ between the points $1$ and  $1-c^2 \e^{2\alpha}$, we have 
$$ |1 - (1-c^2 \e^{2\alpha})^\frac{d}{2} | \leq c(d) c^2 \e^{2\alpha} = o(\e) \quad \text{as} \,\,2 \alpha = 2 (1 - \frac{d}{p}) > 1.$$ 
Therefore 
$$\cL^d \left( D_{\e} \setminus \cF_{\e} \right)=o(\e^{d+1}).$$
We now prove the second part of \eqref{eq:intermediate1_prop3.3br}. We have
\begin{equation}
 \begin{split}
  & \Big| \int_{\cF_{\e}}\rho'\left(\frac{\vert z \vert}{\epsilon}\right)\Pi\left(\frac{(z,f(z))}{\vert (z,f(z)) \vert}\right) \, d \L^d(z) -  \int_{D_{\e}}\rho'\left(\frac{\vert z \vert}{\epsilon}\right)\Pi\left(\frac{(z,f(z))}{\vert (z,f(z)) \vert}\right) \, d \L^d(z) \Big|
 \\& \leq \cL^d\left( D_{\e} \setminus \cF_{\e} \right) \|\Pi\| \| \rho'\|_{\infty} = o(\e^{d+1}).
 \end{split}
\end{equation}
This ends the proof of \eqref{eq:intermediate1_prop3.3br} and the lemma.
\end{proof}
%

%
We now give the proof of Propositions \ref{propbr3.3} and \ref{propbr3.3_gen1}. Compared to the proof of \cite{br}, our proof relies less on the $C^2$ character of the submanifold $M$ and more on the fact that $M$ has a mean curvature in the sense of varifolds with a good integrability, hence it enjoys the tilt-excess decay and height-excess decay properties.
\begin{proof}[Proof of Propositions \ref{propbr3.3} and \ref{propbr3.3_gen1}]\noindent

\noindent Let 
 \begin{equation}
 \begin{array}{ll}
 F : & T \rightarrow \R^n\\
  & z \mapsto (z,f(z)).
\end{array}
\end{equation}
be a local parametrization of $M$ near $0$. $H(\cdot,M) \in L^{\infty}$ and $(\delta M)_s \equiv 0$; therefore Lemma \ref{lemma1_prop3.3br} is valid for $V=M$.


\noindent {\bf{Sketch of the proof}:} We start by showing that $\lm H_{\e}^T=0$, this uses the smallness of $f$ near $0$ and the symmetry of the map $\rho'_{\e}$ w.r.t to the origin. We then use the integral character of $M$ to deduce 
$$\lim\limits_{\e \rightarrow 0} H_{\epsilon}^{S} = \lim\limits_{\e \rightarrow 0} H_{\epsilon}^{T^{\perp} \circ S} =  H(0,M).$$
We show that $\lim\limits_{\e \rightarrow 0} H_{\epsilon}^{T^{\bot} \circ S} = \lm H_{\e}^{2T^{\bot}}$ using the tilt-excess decay estimate and the Taylor expansion of $f$ to the second order, we then obtain the proof of Proposition \ref{propbr3.3} using the linearity of the map $\Pi \mapsto H_{\e}^{\Pi}$. Finally, we prove that $\lm H_{\e}^{S \circ T^{\perp}} =0$ and obtain the proof of Proposition \ref{propbr3.3_gen1}.


\noindent {\bf Step 1:} We show  that $\lm H_{\e}^T=0$.

\noindent Indeed, from \eqref{eq:intermediate1_prop3.3br} we infer that
\begin{equation}
 \lm H_{\e}^{T} = - \lm {\e}^{-d-1}\int_{|z| \leq \e}\rho'\left(\frac{\vert z \vert}{\epsilon}\right)\frac{(z,0)}{\vert (z,f(z)) \vert} \, d \L^d(z).
\end{equation}
 
\noindent The map $f$ is $C^2$, $f(0)=Df(0)=0$ thus $f(z) = O(|z|^2)=O(\e^2)$, this implies 
\begin{equation}
|(z,f(z))|^{-1}=|z|^{-1}\left( 1 + O(|z|^2) \right)= |z|^{-1} + o(\e). 
\end{equation}
Therefore, 
\begin{equation}
 \lm H_{\e}^{T} = - \lm {\e}^{-d-1}\int_{|z| \leq  \e}\rho'\left(\frac{\vert z \vert}{\epsilon}\right)\frac{(z,0)}{|z|} \, d \L^d(z).
\end{equation}
By the symmetry of the map $z \mapsto \rho'\left(\frac{\vert z \vert}{\epsilon}\right)$ w.r.t the origin, we obtain that $\lm H_{\e}^T=0$.

\noindent {\bf Step 2:} We show that $$\lim\limits_{\e \rightarrow 0} H_{\epsilon}^{S} = \lim\limits_{\e \rightarrow 0} H_{\epsilon}^{T^{\perp} \circ S} =  H(0,M).$$

\noindent Indeed, $M$ is $C^2$ hence the varifold associated to $M$  is rectifiable, \cite[Theorem 4.3]{blm1} asserts that 
\begin{equation}
 \lm H_{\epsilon}^{S}=H(0,M). 
\end{equation}
Moreover, the varifold associated to $M$ is integral, \cite[Theorem 5.8]{brakke} asserts that  
$H^{\perp} = H$, noting that $\lm  H_{\epsilon}^{T^{\perp} \circ S} = T^{\perp}(\lm  H_{\epsilon}^{ S}) $ finishes the proof.

\noindent {\bf Step 3:} We show that 
\begin{equation}
 \lim\limits_{\e \rightarrow 0} H_{\epsilon}^{T^{\bot} \circ S} = \lm H_{\e}^{2T^{\bot}}.
\end{equation}

\noindent From \eqref{S_expansion} we have 
$$ T^{\bot}\circ S = 
 \begin{pmatrix}
  0 & 0 \\
  Df & 0
 \end{pmatrix} + O(\norm{Df}^2)I_n.$$
The tilt-excess decay estimate (Theorem \ref{thm:tilttex-heightex}) together with \eqref{eq:pre_tilt_ecess} imply:  
\begin{equation}\label{eq:step3_gen1prf}
  \int\limits_{|y| \leq \e}  \norm{Df}^2 \, d\V  \lesssim \int\limits_{|y| \leq \e}  \|S-T\|^2 \, d\V = o(\e^{d+3}),
 \end{equation}
 thus, we only need  to prove  $\lim\limits_{\e \rightarrow 0} H_{\epsilon}^{Z} = \lm H_{\e}^{2T^{\bot}}$ where
 $Z=\begin{pmatrix}
  0 & 0 \\ Df & 0
 \end{pmatrix}$. \\
 From \eqref{eq:intermediate1_prop3.3br} and denoting $\mathcal{F}_{\e} := \{ z \in T, \, |(z,f(z))|\leq \e \}$ we obtain
 \begin{equation*}
\begin{split}
    &\lm\abs{H_{\epsilon}^{Z} - H_{\e}^{2 T^{\perp}}} 
    \\& = \lm \e^{-d-1}
    \abs{ \int_{\mathcal{F}_{\e}}\rho'\left(\frac{\vert z \vert}{\epsilon}\right) \left(\frac{(0,(Df(z))(z)}{\vert (z,f(z)) \vert}\right)\, d \L^d(z) - \int_{\mathcal{F}_{\e}}\rho'\left(\frac{\vert z \vert}{\epsilon}\right) \left(\frac{(0,2f(z))}{\vert (z,f(z)) \vert}\right) \, d \cL^d(z)  }
 \\& \leq  \lm \e^{-d-1} \|\rho'\|_{\infty} \int_{\mathcal{F}_{\e}} \frac{|(Df(z))(z)-2f(z)| }{\vert (z,f(z)) \vert }\, d\L^d(z). 
\end{split}
 \end{equation*}
Combining the two Taylor expansions of $f$, at $0$:
 $$ f(z)=\frac{1}{2} \left((D^2f)(0)\right)(z,z)+o(\abs{z}^2),$$
and at $z$:
 $$ f(0)= f(z) + (Df(z))(-z)+\frac{1}{2}\left((D^2 f)(z)\right)(-z,-z)+o(\abs{z}^2)$$
we get: $$|(Df(z))(z)-2f(z)| =o(|z|^{2}) $$ and this finishes the proof of Step $3$.
 
\noindent {\bf Step 4:} Proof of Proposition \ref{propbr3.3}.
So far, we showed that $\lm H_{\e}^{\Pi} = H(0,M)$ for $\Pi \in \{S, 2 T^{\perp}, T^{\perp}\circ S \}$ and $\lm H_{\e}^{T} = 0$. The map $\Pi \rightarrow H_{\e}^{\Pi}$ is linear, from 
\begin{equation}
 2 {\rm Id} = 2T^{\perp} + 2T, \quad \text{and} -2S^{\perp}= 2S -2{\rm Id}
\end{equation}
and the orthogonality of $H(0,M)$, we deduce that $\lm H_{\e}^{\Pi} = H(0,M)$ for $\Pi \in \{2{\rm Id}, -2S^{\perp} ,  -2T^{\perp}\circ S^{\perp} \}$
and this finishes the proof of Proposition \ref{propbr3.3}.

\noindent {\bf Step 5:} We show that
\begin{equation}
 \lm H_{\e}^{S \circ T^{\perp}} =0.
\end{equation}

\noindent From \eqref{S_expansion} we have 
 $$ S\circ T^{\bot} = \begin{pmatrix}
  0 & Df^t \\ 0 & 0
 \end{pmatrix} + O(\norm{Df}^2)I_n, $$ 
\eqref{eq:step3_gen1prf} implies 
\begin{equation}\label{STBAR}\begin{split}
\lm  \e^{-d-1}\abs{H_{\e}^{S\circ T^{\bot}}} & 
= \lm \Bigl( \e^{-d-1} \abs{ \int_{\mathcal{F}_{\e}}\rho'\left(\frac{\vert z \vert}{\epsilon}\right) 
\begin{pmatrix}
0 & Df^t \\ 0 & 0
\end{pmatrix} 
\left(\frac{(z,f(z))}{\vert (z,f(z)) \vert}\right) \, d \L^d(z)} +o(\e^2) \Bigr)
\\ &
\leq \| \rho'\|_{\infty} \lm \e^{-d-1} \int_{\mathcal{F}_{\e}} \frac{ |  Df(z)^t(f(z)) |}{|(z,f(z))|} \, d\L^d(z).
\end{split}
\end{equation}
As $f(0)=0$ and $Df(0)=0$, we can infer that $|f(z)|=O(|z|^2)$ and $\| Df(z) \| = O(|z|)$. This finishes the proof of step 5.

\noindent \noindent {\bf Step 6:} Proof of Proposition \ref{propbr3.3_gen1}.

\noindent From Step $4$ and the orthogonality of $H$ we infer that $\lm H_{\e}^{\Pi}= 0$ for $\Pi \in \{ T\circ S , T \circ S^{\perp} \}$. We proved so far, in step $1$ and  step $4$, that $\lm H_{\e}^{\Pi} = 0$ for $\Pi \in \{T, S \circ T^{\perp} \}$, that concludes the proof of the second part of the proposition. 

\noindent Finally, from
\begin{equation}
 S^{\bot}\circ T^{\bot}+S\circ T^{\bot} = T^{\bot}, \quad  S\circ T^{\bot}+S\circ T = S, \quad \text{and} \,\,S^{\bot}\circ T+S\circ T = T,
\end{equation}
and step $4$ we conclude that $\lm H_{\e}^{\Pi}= H$ for $\Pi \in \{2S^{\bot}\circ T^{\bot}, S\circ T,   -S^{\bot}\circ T \}$ and the proof of Proposition \ref{propbr3.3_gen1}.
\end{proof}

\begin{remk}[$C^3$-Regularity and convergence rate]
In case the submanifold $M$ is  $C^3$, and as mentioned in \cite{br}, we have a better convergence rate. In fact, if $f\in C^3$ then the rest in the Taylor expansion of $f$ is of the order $O(|z|^3)$ instead of $o(|z|^2)$ when $f$ is only $C^2$.
\end{remk}

\begin{remk}(Use of natural kernel pair)
As we noticed in the proof and by reading the proof of \cite[Proposition 3.3]{br}, the kernels $\rho$ and $\xi$ do not need to be natural kernel pair \eqref{nkpcondition}. 
\end{remk}

\begin{prop}[Generalization in terms of regularity]\label{propbr3.3_gen2} Let $V$ be an $(n-1)$-integral varifold in $\R^{n}$ with unit density and $H(\cdot,V) \in \cL^p(\V)$ with $p>2(n-1)$ and $(\delta V)_s \equiv 0$.\\
For $\Pi$ in the set:  
$$ \lbrace S , -2S^{\bot}  , 2{\rm Id}, T^{\bot}\circ S, -2T^{\bot}\circ S^{\bot}, 2T^{\bot} \rbrace,$$
or in :
$$ \lbrace 2S^{\bot} \circ T^{\bot}, S\circ T, -S^{\bot} \circ T \rbrace $$
one has : $ \lim\limits_{\e \rightarrow 0} H_{\epsilon}^{\Pi} = H.$
For $\Pi$ in the set:
\begin{equation}
 \lbrace T, T\circ S, S\circ T^{\bot}, T\circ S^{\bot} \rbrace
\end{equation}
one has
\begin{equation}
 \lm H_{\e}^{\Pi}=0.
\end{equation}
\end{prop}
\begin{proof}
Denote $D_{\e} := \{ z \in T, |z|\leq \e \}$. Let 
 \begin{equation}
 \begin{array}{ll}
 F : & T \rightarrow \R^n\\
  & z \mapsto (z,f(z))
\end{array}
\end{equation}
be a local parametrization of $\supp \|V\|$ near $0$, it exists thanks to Theorem \ref{thm:allard_reg}. From Lemma ~\ref{poly_approx}, we can affirm the existence of a polynomial function $q$ on $T$ satisfying: $ \sup_{D_{\e}} |f-q|=o(\e^2)$. Denote by $W$ the varifold associated to the graph of $q$ near $0$ and by  
\begin{equation}
 \begin{array}{ll}
 Q : & T \rightarrow \R^n\\
  & z \mapsto (z,q(z)).
\end{array}
\end{equation}
a local parametrization of $\supp \|W\|$ near $0$.

\noindent The idea of the proof is to show that
$$\lm H_{\e}^{\Pi_V}(0,V) = \lm H_{\e}^{\Pi_W}(0,W)$$
$\Pi_V$ and $\Pi_W$ have the same form (to be explained) but the first depends on $V$ and the second on $W$. 

\noindent We note that $T_0W=T_0V := T$ as $\nabla q(0) =0$, this comes from the fact that $f$ and $q$ are $C^1$, $\nabla f (0)=0$ and $|f-q|=o(\e^2)$. 

\noindent {\bf Step 1:} We show that 
\begin{equation}
 \|T_{(z,f(z))}V - T_{(z,q(z))}W  \| \lesssim \| \nabla f(z) - \nabla q(z) \| \,\,\text{near $0$}. 
\end{equation}

\noindent We first recall the expressions of the matrices associated to the projections on the tangent spaces (dropping the variable $z$ for simplicity), we have
\begin{equation}
 T_{(z,f(z))}V = DF(DF^tDF)^{-1}DF^t , \quad \text{and} \,\, T_{(z,q(z))}W = DQ(DQ^tDQ)^{-1}DQ^t
\end{equation}
We know that $DF^t DF = I_{n-1} + o(1)$ and $DQ^t DQ = I_{n-1} + o(1)$, hence 
$$ \| DF^t DF\| \lesssim 1,  \quad  \, \text{and} \, \| DQ^t DQ\| \lesssim 1.$$ 
Using $\| A^t \| = \|A\|$ for any matrix $A$ together with $\| DF \| \lesssim 1 $ and  $\| DQ \| \lesssim 1$ we infer 
\begin{equation}\label{eq:prop3.3br_gen2_prf1}
\begin{split}
 \| (DF^t DF)^{-1} - (DQ^tDQ)^{-1} \|
 & \leq  \|(DF^t DF)^{-1} \| \|(DQ^tDQ)^{-1} \| \|  DF^t DF - DQ^tDQ \|
 \\& \lesssim \|  DF^t DF - DQ^tDQ \| \leq \| DF^t \| \| DF - DQ \| + \| DQ \|\| \| DF^t - DQ^t \|
 \\& \lesssim \| DF - DQ \| = \| \nabla f - \nabla q \|.
 \end{split}
\end{equation}
From \ref{eq:prop3.3br_gen2_prf1} we obtain 
\begin{equation}
 \begin{split}
  \|T_{(z,f(z))}V - T_{(z,q(z))}W  \|
  & = \| DF(DF^tDF)^{-1}DF^t - DQ(DQ^tDQ)^{-1}DQ^t \|
  \\& \leq \| DF(DF^tDF)^{-1} \| \| DF^t - DQ^t \| 
  \\& + \| DF \| \|DQ\| \| (DF^t DF)^{-1} - (DQ^tDQ)^{-1} \|
  \\& + \| (DQ^tDQ)^{-1}DQ^t \| \| DF-DQ\|
  \\& \lesssim \|DF-DQ\| = \| \nabla f - \nabla q \|.
 \end{split}
\end{equation}
This finishes the proof of step $1$.

\noindent {\bf Step 2:} We show that 
\begin{equation}
 \int_{D_{\e}} | \nabla f - \nabla q | \, d\cL^d = o(\e^{d+1}). 
 \end{equation}

\noindent Indeed, in the sense of distributions,  we have 
 \begin{equation} 
 \begin{split}
     H(\cdot,V)  &= \text{div}\left(\frac{\nabla f}{(1+|\nabla  f|^2)^\frac{1}{2}}\right)
     \end{split}
     \end{equation}
and,
     \begin{equation}
 \begin{split}
     H(\cdot,W)  &= \text{div}\left(\frac{\nabla q}{(1+|\nabla q|^2)^\frac{1}{2}}\right)
     \end{split}
     \end{equation}
We set : $h=H(\cdot,V)-H(\cdot,W)$ and $\phi=f-q$. Our goal now  is to show that 
\begin{equation}
 \int_{D_{\e}}  |\nabla  \phi| \, d\cL^d = o(\e^{d+1}).
\end{equation}
Define a function $\eta$ on $T$ such that $\eta = 1$ on $D_{\e}$, $0$ outside $D_{2\e}$ and $|\nabla\eta| \leq  2 \e^{-1}$. We test the equation satisfied by $h$ against the function $\phi \eta^2$, we obtain using $|\nabla  f|=o(1)$ and $|\nabla  q|=o(1)$ that 
 \begin{equation}
     \begin{split}
         \int_{D_{2\e}} h \phi \eta^2 \, d\cL^d &= -\int_{D_{2\e}}  (1+o(1)) \nabla  f \cdot \nabla (\phi \eta^2 ) \, d\cL^d - \int_{D_{2\e}}  (1+o(1)) \nabla  q \cdot \nabla (\phi \eta^2 ) \, d\cL^d
         \\& =-(1+o(1)) \left(  \int_{D_{2\e}} \nabla  \phi \cdot \nabla (\phi \eta^2 ) \, d\cL^d \right)
         \\& = -(1+o(1)) \left(  \int_{D_{2\e}} |\nabla \phi|^2   \eta^2 \, d\cL^d  \right) - (1+o(1)) \left(  \int_{D_{2\e}} \phi  \nabla  \phi \cdot \nabla \eta \eta  \, d\cL^d \right).
     \end{split}
 \end{equation}
Using $\phi=o(\e^2)$ and $h\in \cL^1$ (as it belongs to $\cL^p$ and $p \geq 2$), we obtain 
 \begin{equation}
     \begin{split}
        &(1+o(1)) \left(  \int_{D_{2\e}} |\nabla \phi|^2   \eta^2  \, d\cL^d \right) 
        \\&= -\int_{D_{2\e}} h \phi \eta^2 \, d\cL^d 
        - (1+o(1)) \left(  \int_{D_{2\e}} \phi  \nabla  \phi \cdot \nabla \eta \eta \, d\cL^d \right)
        \\& \leq o(\e^{d+2}) + (1/2+o(1)) \int_{D_{2\e}} |\nabla \phi|^2 \eta^2 \, d\cL^d + (1/2+o(1)) \int_{D_{2\e}} \phi^2 |\nabla \eta|^2 \, d\cL^d
     \end{split}
 \end{equation}
 where we used the inequality $ 2ab\leq a^2+b^2$. Using $\phi=o(\e^2)$ and $|\nabla \eta|\leq 2 \e^{-1}$ we obtain 
 \begin{equation}
   \int_{D_{\e}}  |\nabla \phi|^2  \, d\cL^d \leq  \int_{D_{2\e}}  |\nabla \phi|^2 \eta^2 \, d\cL^d = o(\e^{d+2})
 \end{equation}
 finally, by the Cauchy-Schwarz inequality, we obtain : 
 \begin{equation}
     \int_{D_{\e}}  |\nabla \phi| \, d\cL^d = o(\e^{d+1}),
 \end{equation}
and this concludes the proof of step $2$.

\noindent {\bf Step 3} We prove Proposition  ~\ref{propbr3.3_gen2}.

\noindent We recall that $T_0W = T_0V := T$, from ~\eqref{eq:intermediate1_prop3.3br} we infer that for any linear operator $\Pi$ one has
$$ -\lm H_{\e}^{\Pi}(0,W)  = 
 \lm {\e}^{-d-1}\int_{D_{\e}}\rho'\left(\frac{\vert z \vert}{\epsilon}\right)\Pi\left(\frac{(z,q(z))}{\vert (z,q(z)) \vert}\right) \,d \L^d(z). $$
 Similarly for $V$, for any linear operator $\Pi$ one has
$$ -\lm H_{\e}^{\Pi}(0,V)  = 
 \lm {\e}^{-d-1}\int_{D_{\e}} \rho'\left(\frac{\vert z \vert}{\epsilon}\right)\Pi\left(\frac{(z,f(z))}{\vert (z,f(z)) \vert}\right) \, d \L^d(z) .$$
 
\noindent Denote $S_X := T_{\cdot}X$ for any varifold $X$ and  
\begin{equation}
 \left( P_{1,X}, \,  P_{2,X},\,  P_{3,X},\,  P_{4,X},\,  P_{5,X} \right) := \left(  {\rm Id},\, S_X,\, S_X^{\perp},\, T ,\, T^{\perp} \right)
\end{equation}
Let $\Pi_V = P_{i,V} \circ P_{j,V}$ and $\Pi_W = P_{i,W} \circ P_{j,W}$ for some $i$ and $j$ in $\{ 1, \dots , 5 \}$, $i \neq j$. Using  $\| A^t \| = \|A\|$ for any matrix $A$, $\| P_{i,X} \| \leq 1$ for any $i\in \{ 1, \dots , 5 \}$ and any varifold $X$ and the triangle inequality, one can infer that 
$$ \| \Pi_{V} - \Pi_{W} \| \leq 2 \| S_V - S_W \|  $$ 
Then, step $1$ and step $2$ imply that 
\begin{equation}\label{eq:similar_projectors}
 \begin{split}
 & \Big|  \lm H_{\e}^{\Pi_V}(0,V)  - \lm H_{\e}^{\Pi_W}(0,W)  \Big|
 \\& = \lm \e^{-d-1} \Big| \int_{D_{\e}}\rho'\left(\frac{\vert z \vert}{\epsilon}\right)\Pi_V \left(\frac{(z,f(z))}{\vert (z,f(z)) \vert}\right) \,d \L^d(z) - \int_{D_{\e}}\rho'\left(\frac{\vert z \vert}{\epsilon}\right)\Pi_W\left(\frac{(z,q(z))}{\vert (z,q(z)) \vert}\right) \,d \L^d(z) \Big|
 \\& \leq \|\rho'\|_{\infty} \lm \e^{-d-1} \int_{D_{\e}}   \| \Pi_{V}(z,f(z)) - \Pi_{W}(z,q(z)) \|   \,d \L^d(z) 
 \\& \leq 2  \|\rho'\|_{\infty} \lm \e^{-d-1} \int_{D_{\e}}   \| S_V(z,f(z)) - S_W(z,q(z)) \|   \,d \L^d(z) =0.
 \end{split}
\end{equation}
From \eqref{eq:similar_projectors} and Proposition ~\ref{propbr3.3_gen1}  (for $M=W$) we obtain that $\lm H_{\e}^{\Pi_V}(0,V) = 0$ for any $\Pi_V$ in 
\begin{equation}
 \lbrace T, T\circ S_V, S_V\circ T^{\bot}, T\circ S_V^{\bot} \rbrace.
\end{equation}
As both $V$ and $W$ are rectifiable, \cite[Theorem 4.3]{blm1} combined with \eqref{eq:similar_projectors} imply
\begin{equation}
 H(0,V) = \lm H_{\e}^{S_V} (0,V ) = \lm H_{\e}^{S_W} (0,W ) = H(0,W).
\end{equation}
Therefore, 
\begin{equation}
 \lm H_{\e}^{\Pi_V}(0,V) = H(0,W) = H(0,V)
\end{equation}
for any $\Pi_{V}$ in the set
\begin{equation}
 \lbrace S ,\, -2S_V^{\bot}  ,\, 2{\rm Id},\, T^{\bot}\circ S_V, \, -2T^{\bot}\circ S_V, \, 2T^{\bot} \rbrace \cup \lbrace 2S_V^{\bot} \circ T^{\bot}, \, S_V\circ T, \, -S_V^{\bot} \circ T \rbrace,
\end{equation}
and we finish the proof of Proposition ~\ref{propbr3.3_gen2}.
\end{proof}
%


%
%

\subsection{Extention to the second fundamental form}\label{sec:sff_app}
The authors of \cite{blm2} suggested a definition of the second fundamental form and its approximation after the work of Hutchinson \cite{hut1}.   We will give a general definition of the approximate second fundamental form in the same spirit of the definition of $H_{\e}^{\Pi}$.
\vspace{0.3cm}

\noindent Let $V \in V_d(\R^n)$, we define the $G$-linear variation $\delta_{ijk}V : C^{1}(\R^n, \R) \rightarrow \R $ as in \cite[Definition 3.1]{blm2} by: 
\begin{equation}
 \delta_{ijk}V(\phi) := \int_{\R^n \times \G} P_{jk} \,P(\nabla \phi) \cdot e_i \,\, dV(y,P), \,\, \text{for any} \, \phi \in C^1_c(\R^n,\R),
\end{equation}
where $i,j,k \in \{ 1, \dots, n \}$. In case $\delta_{ijk}$ is bounded (as a measure) for every $i,j,k \in, \{1, \dots, n \}$ we can use the Riesz representation Theorem followed by the Radon-Nikodym decomposition to infer the existence of a tensor $\beta^V_{ijk}$ satisfying:
\begin{equation}
 \delta_{ijk} = - \beta^V_{ijk} \V + (\delta_{ijk}V)_s
\end{equation}
where the measure $(\delta_{ijk}V)_s$ is singular w.r.t $\V$. For any $\e\in(0,1)$, the approximate second fundamental form tensor $\left(A_{ijk}^{V,\e}\right)_{ijk}$ is then given by the formula (\cite[Definition 6.2]{blm2})
\begin{equation}
 A_{ijk}^{V,\e} := \beta_{ijk}^{V,\e}-c_{jk}^{V,\e} \left(\left( I+c^{V,\e}\right)^{-1}H_{\e}(\cdot,V)\right)_{i}
\end{equation}
where
\begin{equation}
 \beta_{ijk}^{V,\e}(x):=-\frac{C_{\xi}}{C_{\rho}}\frac{\delta_{ijk}V*\rho_{\e}(x)}{\V*\xi_{\e}(x)}= -\frac{C_{\xi}}{C_{\rho}}\frac{\e^{-d-1}\int_{\R^n\times \G}P_{jk}\rho'\left(\frac{\vert y-x \vert }{\e} \right) P\left( \frac{ y-x}{ \vert y-x \vert}\right) \cdot e_i dV(y,P)}{ \e^{-d}\,\V*\xi_{\e}(x)},
\end{equation}
\begin{equation}
 c_{jk}^{V,\e} := \frac{\left( \int_{\G} P_{jk} \, d \nu_{\cdot}(P) \right) \ast \eta_{\e} }{\V \ast \eta_{\e}}
\end{equation}
and
\begin{itemize}
 \item $\eta \in C^0(\R^+,\R^+)$, positive on $(0,1)$ and supported on $[0,1]$; and
 $$ \eta_{\e}(x) := \eta\left( \frac{|x|}{\e} \right) \,\, \forall x \in \R^n.$$ 
 \item $C_{\rho}$ and $C_{\xi}$ are normalization constants \eqref{def:cxi_crho},
 \item $P$ being the projection matrix on the subspace $P \in \G$,
 \item $\nu_{\cdot}$ is the Grassmann component of $V$.
\end{itemize}
Similarly to the work \cite{br} on the approximate mean curvature, we suggest the following definition:

\begin{dfn}[Generalized approximate second fundamental form]

\noindent Let $V\in V_d(\R^n)$, $\e\in(0,1)$. We define a generalized approximate second fundamental form by 
\begin{equation}
 A_{ijk}^{V,\Pi,\e} := \beta_{ijk}^{V,\Pi,\e}-c_{jk}^{V,\e} \left(\left( I+c^{V,\e}\right)^{-1}H_{\e}(\cdot,V)\right)_{i}
\end{equation}
where
$$ \beta_{ijk}^{V,\Pi,\e}(x) =  -\frac{C_{\xi}}{C_{\rho}}\frac{\e^{-d-1}\int_{\R^n\times \G}P_{jk}\rho'\left(\frac{\vert y-x \vert }{\e} \right) \Pi_y \left( \frac{ y-x}{ \vert y-x \vert}\right) \cdot e_i \,dV(y,P)}{\e^{-d} \, \V*\xi_{\e}(x)},$$
and $\Pi_{y}$ is a linear operator for any $y\in \supp \|V\|$.
\end{dfn}

We recall the notations $T := T_0V$ and $S=T_yV$. In \cite[Proposition 6.7]{blm2} it was proven that $\lm \beta_{ijk}^{V, S,\e} = \beta_{ijk}^{V}$ in the rectifiable setting. In the following proposition, we show the convergence of $\beta_{ijk}^{V,\Pi}$ to $\beta_{ijk}^{V,\e}$ for  $\Pi=S\circ T^{\perp}$ for unit density varifolds with some integrability condition on the mean curvature similar to Proposition ~\ref{propbr3.3_gen2}.

\begin{prop}[Approximation of the SFF]
Let $V$ be a $(n-1)$-integral unit density varifold in $\R^{n}$, assume that $ (\delta V)_s \equiv 0, \, H(\cdot,V) \in \cL^p(\V) $ with $p > 2(n-1)$.
One has,
$$ \lm \beta_{ijk}^{V,S,\e} = \beta_{ijk}^{V}, \quad \lm \beta_{ijk}^{V,S\circ T,\e} = \beta_{ijk}^{V} \quad \text{and} \quad
\lm \beta_{ijk}^{V,S\circ T^{\bot},\e} = 0.$$
\end{prop}
\begin{proof}[Elements of the proof]
By linearity of the map $\Pi \mapsto \beta_{ijk}^{V,\Pi,\e}$, it is enough to show that: 
\begin{equation*}
 \lm \beta_{ijk}^{V,S\circ T^{\bot},\e} = 0 \quad (\text{as} \lm \beta_{ijk}^{V,S,\e} = \beta_{ijk}^{V}).
\end{equation*}
To do so, and as done before for the mean curvature,
we set $x=0$ and eliminate the denominator and the constant $C_{\rho}$ to have the following simpler formula:
$$ \lm \beta_{ijk}^{V,S \circ T^{\perp},\e} = - \lm \e^{-d-1}\int\limits_{\R^n \times \G}P_{jk}\rho'\left(\frac{\vert y \vert }{\e} \right) S \circ T^{\perp} \left( \frac{ y}{ \vert y \vert}\right) \cdot e_i \,\,dV(y,P). $$
We notice that as $V$ is rectifiable, $V = \V \otimes \delta_S$ and  
$$\int\limits_{\R^n \times \G}P_{jk}\rho'\left(\frac{\vert y \vert }{\e} \right) S \circ T^{\perp} \left( \frac{ y}{ \vert y \vert}\right) \cdot e_i \,dV(y,P) = \int\limits_{\R^n }S_{jk}\rho'\left(\frac{\vert y \vert }{\e} \right) S \circ T^{\perp} \left( \frac{ y}{ \vert y \vert}\right) \cdot e_i \,d\|V\|(y).$$

\noindent In the $C^2$ case, similarly to the proof of Propositions ~\ref{propbr3.3} and ~\ref{propbr3.3_gen1}, we use the area formula and the height-excess decay to write down the integral w.r.t $z(\in T)$ and get rid of the Jacobian of the parametrization in the integral. We then use the identity \eqref{S_expansion}, \eqref{eq:intermediate1_prop3.3br} and the tilt-excess decay to conclude.\\
In the general case, similarly to the proof of Proposition \ref{propbr3.3_gen2} we use the polynomial approximation lemma (Lemma \ref{poly_approx}) to prove that, denoting $S_V=T_{y}V$ and $S_W=T_yW$ for simplicity
\begin{equation}
\beta_{ijk}^{V,S_V\circ T^{\bot},\e} = \beta_{ijk}^{V,S_W\circ T^{\bot},\e} = 0
\end{equation}
where $W$ is the varifold associated to the polynomial function approximating the graph of $\supp \V$ near $0$.
\end{proof}
%


\section{Comparison principles for the continuous and discrete motions of point clouds} \label{sec:comp_ppl}


The authors of \cite{br} constructed several schemes for motions of points clouds by approximate mean curvature and proved sphere comparison principles to internal varifolds. We will recall the definitions and the result of \cite{br} on and provide some extensions.

\noindent In the sequel, $\rho$ and $\xi$ are assumed to be natural kernel pair \eqref{nkpcondition}, a quick computation shows that $ \frac{C_{\rho}}{C_{\xi}}=-\frac{d}{n}$, this justifies formulas \eqref{def:cont_motion_exp}, \eqref{def:disc_motion1_exp2} and \eqref{def:disc_motion2_exp2} .


\subsection{Continuous motion of points cloud varifolds}


Given a points cloud $d$-varifold $V=\sum\limits_{i=1}^{N}m_i\delta_{(x_i,P_i)}$ in $\R^n$; a continuous motion of points cloud varifolds by approximate mean curvature flow starting from $V$ is a family of varifolds $\left(V(t)\right)_{t\geq0}$ with $V(0)=V$ and:
$$ V(t)=\sum\limits_{i=1}^{N} m_i(t)\delta_{(x_i(t),P_i(t))} \quad \text{and} \quad X(t)=\left( x_1(t) \dots x_N(t)\right) \in \R^{nN}$$
such that
$$ \frac{d}{dt}x_i(t)=H_{\e}^{\Pi}(x_i(t),V(t)) $$
where $H_{\e}^{\Pi}$ the approximate mean curvature associated to the linear operator $\Pi$ \eqref{comp:def_mc}. The evolution equation turns into : 
\begin{equation}\label{def:cont_motion}
    \frac{d}{dt} x_i(t)=\frac{1}{\e} \sum\limits_{j=1}^N \omega_{ij}(t) \Pi_{ij}(t) (x_j(t)-x_i(t)), \quad i=1 \ldots N
\end{equation}
where
\begin{equation}\label{def:cont_motion_exp}
\omega_{ij}(t)=-\frac{d}{n}\frac{m_j(t)\rho'\left(\frac{|x_j(t)-x_i(t)|}{\e}\right)\frac{1}{|x_j(t)-x_i(t)|}}{\sum\limits_{l=1}^N m_l(t)\xi\left(\frac{|x_j(t)-x_i(t)|}{\e}\right)} \,\, \text{for} \,\, i \neq j, \,\, \omega_{ii}=0,
\end{equation}
and the $\Pi_{ij}$'s are the coefficients of the matrix associated with $\Pi$.\\

\noindent  The existence of such motion may not be guaranteed in every case, it depends on the definition of the masses and the tangents. In case the previous quantities are Lipschitz functions on the positions of the points, the motion exists for at least a short time interval. The proofs provided by Buet and Rumpf in \cite{br} are valid for any definition of $\{m_i\}_i$ and $\{P_i\}_i$ throughout the evolution, hence we keep these definitions implicit.


\noindent We start by recalling \cite[Proposition 5.2]{br} on the sphere barrier to internal varifolds principle satisfied by the continuous motion.
\begin{prop}[Sphere barrier to internal varifolds, Proposition 5.2 \cite{br}]\noindent \label{prop:comparison_intvar}
\noindent Let $(V(t))_{t\geq 0}$ be a continuous motion of points clouds defined by ~\eqref{def:cont_motion}, for $\Pi = 2{\rm Id}$
one has 
\begin{equation}
 \supp \|V(0)\| \subset B(a,R) \quad \implies \supp \|V(t)\| \subset B(a,\sqrt{R(0)^2-2dt}), \, \forall t \in [0,R(0)^2/2d].
\end{equation}
\end{prop}

The continuous motion also satisfies the sphere barrier to external varifolds principle, it is stated as follows:
\begin{prop}[Sphere barrier to external varifolds] \label{prop:comparison_extvar} Let $(V(t))_{t\geq 0}$ be a continuous motion of points clouds defined by ~\eqref{def:cont_motion}, for $\Pi = 2{\rm Id}$
one has 
\begin{equation}
 \supp \| V(0)\| \subset B(a,R)^c \quad \implies \supp \|V(t)\| \subset B(a,\sqrt{R(0)^2-2dt})^c, \, \forall t \in [0,R(0)^2/2d].
\end{equation}
\end{prop}  

\begin{proof}
The proof is similar to the proof of the sphere barrier to internal varifolds principle, we just interchange $\max$ and $\min$ in  the definitions of $c(t)$ and $R(t)$ (cf. proof of  \cite[Proposition 5.2]{br}).

\noindent Without loss of generality, one can assume that $z=0$. 
For a points cloud $X^0 = \{ x_i^0 \}_{i=1}^N \subset \R^n$ define $r^0 = \min_{i=1,\ldots,N} |x_i^0 |$ and assume that
$(X(t))_{0 \leq t <T}$ is a continuous motion by approximate mean curvature with $X(0) = X^0$. We first prove that $r(t) = \min_{i=1,\ldots,N} |x_i(t)|$ fulfills
\[r(t) \geq \sqrt{(r^0)^2 - 2d \int_0^t c(s) \mathrm{d}s}
\]
with 
\begin{equation} \label{eq:Pt}
c(t) = \max \left\lbrace \left. \frac{ \Pi_{ij}(t)(x_i(t) - x_j(t)) \cdot x_i(t) }{| x_i(t) - x_j(t) |^2  } \: \right| \: \begin{array}{l}
i \in \{1, \ldots, N \}, \: \displaystyle |x_i(t) | = r(t)  \text{ and } \\
j \in \{1, \ldots, N \}, \: \displaystyle 0< | x_i(t) - x_j(t) | <  \epsilon 
\end{array}   \right\rbrace \: .
\end{equation}
where $\Pi$ is general.\\

\noindent Indeed, choose $i \in \{ 1, \dots, N\} $ with $r(t) = |x_i(t)|$. We have : 
\begin{align*}
\tfrac12 \tfrac{\mathrm{d}}{\mathrm{d}t} |x_i(t)|^2 &= \tfrac{\mathrm{d}}{\mathrm{d}t} x_i(t) \cdot x_i(t) 
= \frac{1}{\epsilon} \sum_{j=1}^N \omega_{ij}(t) \Pi_{ij}(t) \left( x_j(t) - x_i(t) \right) \cdot x_i(t) \\
&\geq -\frac{c(t)}{\epsilon} \sum_{j=1}^N \omega_{ij}(t) |x_i(t)-x_j(t) |^2 \\
&=  \frac{c(t)d}{n} \sum_{j=1}^N \frac{ \displaystyle m_j(t) \rho^\prime \left(\frac{|x_j(t) - x_i(t) |}{\epsilon} \right) \frac{| x_j(t) - x_i(t) |}{\epsilon} }{\displaystyle \sum_{l=1}^N m_l(t) \xi \left( \frac{|x_l(t) - x_i(t) |}{\epsilon} \right)} 
= - c(t) d\:,
\end{align*} 
where we used \eqref{nkpcondition}. Integrating w.r.t $t$  we obtain
\begin{equation}\label{eq:pre_extvar}
 r(t)^2 \geq (r^0)^2 - 2d \int_0^t c(s) \mathrm{d}s
\end{equation}
which proves the claim.
\end{proof}

\noindent In case $\Pi=2 {\rm Id}$ one can prove that $c(t)\leq 1$. Indeed
\begin{equation}
(x_i-x_j)\cdot(x_i)=|x_i|^2-(x_j)\cdot(x_i)\leq \frac12|x_i|^2 - (x_j)\cdot(x_i) +\frac{1}{2}|x_j|^2 =\frac12|x_i-x_j|^2.
\end{equation}
Hence $c(t)\leq 1$ and 
$$ r(t)^2 \geq (r^0)^2 - 2dt $$ 
this finishes the proof of Proposition ~\ref{prop:comparison_extvar}.
From the sphere barrier principles for external and internal varifolds, we obtain the following result:
\begin{cor} If a $d$-points cloud varifold is contained in a sphere of radius $R(0)$, then, its continuous motion by the approximate mean curvature (when it exists) is contained in the evolution of the sphere of radius $R(t)$ with
\begin{equation*}
 R(t)^2=R(0)^2-2dt.
\end{equation*}
\end{cor}

\subsection{Discrete (in time) motions of points cloud}
We now consider the time discretizations of the equation \eqref{def:cont_motion} and the comparison principles of the associated motions. Let us consider $\tau >0$ to be the time step, we start with the implicit schemes (implicit with respect to the positions) and we propose (following \cite{br}) the following schemes.
\begin{align}\label{def:disc_motion1}
 x_i^{k+1}=x_i^{k}+\tau H_{\e}^{\Pi} \left(x_i^{k+1},V\right)
\end{align}
where
\begin{equation}
 V=\sum\limits_{j=1}^{N}m_i^{k} \delta_{(x_i^{k+1},P_i^k)} \quad \text{or} \quad  V=\sum\limits_{j=1}^{N}m_{i}^{k+1} \delta_{(x_i^{k+1},P_i^k)} 
\end{equation}
$X^k := \left( x_1^k,\dots,x^k_N \right) \in \R^{nN}$ being the positions at time $t_k := k\tau$. One could also consider the linear operator $(P_i^k)_{i,k}$ to be implicit and the comparison principle below would still hold, to make it simple we keep them explicit. The evolution equation, the case where the masses and the projectors are explicit and the positions implicit can be written as follows: for $i \in \{ 1, \dots, N \}$
\begin{equation}\label{def:disc_motion1_exp1}
 x_i^{k+1} = x_i^k + \frac{\tau}{\e} \sum_{j=1}^{N} \omega_{ij}^{k} \Pi_{ij}^k\left( x_j^{k+1} - x_i^{k+1} \right)
\end{equation}
with 
\begin{equation}\label{def:disc_motion1_exp2}
\omega_{ij}^k=-\frac{d}{n}\frac{m_j^k\rho'\left(\frac{|x_j^k-x_i^k|}{\e}\right)\frac{1}{|x_j^k-x_i^k|}}{\sum\limits_{l=1}^N m_l^k\xi\left(\frac{|x_j^k-x_i^k|}{\e}\right)} \quad \text{for} \,\, i \neq j, \,\, \omega_{ii}^k=0.
\end{equation}
In order to find the expression of a scheme that is implicit in the masses or in the linear operator, it only suffices to change $k$ into $k+1$ in $m_j^k$ in   \eqref{def:disc_motion1_exp2}, for the linear operators we  change  $k$ into $k+1$ in $\Pi_{ij}^k$ in \eqref{def:disc_motion1_exp1}. The following result stems from \cite[Proposition 5.4]{br}, as the reader may notice, the proof does not require anything on the masses, hence it is valid for both explicit and implicit definitions of the masses.

\begin{prop}[Sphere barrier to internal varifolds]\label{prop:comparison_intvar_disc}
 Let $V$ be a $d$-points cloud contained in a ball $B(z,R(0)), \, z \in \R^n$. Assume that $(V^k)_k$ is a sequence of points cloud varifolds solution to the equation \eqref{def:disc_motion1} for $\Pi=2 { \rm Id}$ and that $V^0:=V$. Set $t_k=k\tau$, then $\supp \|V^k\| \subset B(z,\sqrt{R(0)^2-2dt_k})$ for any $k\in N$ such that $R(0)^2 >2dt_k$.
\end{prop}

Though not interesting from a numerical point of view, as it lacks stability, we still introduce the explicit (in positions) scheme and highlight the comparison principle satisfied by this motion. We consider the following equation:

\begin{equation}\label{def:disc_motion2}
 x_i^{k+1}=x_i^{k}+\tau H_{\e}^{\Pi} \left(x_i^{k},V\right)
\end{equation}
with 
\begin{equation}
V=\sum\limits_{j=1}^{N}m_i^{k} \delta_{(x_i^{k},P_i^k)} \quad \text{or} \quad V=\sum\limits_{j=1}^{N}m_{i}^{k+1} \delta_{(x_i^{k},P_i^k)}.
\end{equation}
\noindent The evolution equation when the positions, the masses and the linear operators are explicit can be written as follows: for $i \in \{ 1, \dots, N \}$
\begin{equation}\label{def:disc_motion2_exp1}
 x_i^{k+1} = x_i^k + \frac{\tau}{\e} \sum_{j=1}^{N} \omega_{ij}^{k} \Pi_{ij}^k\left( x_j^{k} - x_i^{k} \right)
\end{equation}
with 
\begin{equation}\label{def:disc_motion2_exp2}
\omega_{ij}^k=-\frac{d}{n}\frac{m_j^k\rho'\left(\frac{|x_j^k-x_i^k|}{\e}\right)\frac{1}{|x_j^k-x_i^k|}}{\sum\limits_{l=1}^N m_l^k\xi\left(\frac{|x_j^k-x_i^k|}{\e}\right)} \quad \text{for} \,\, i \neq j, \,\, \omega_{ii}^k=0.
\end{equation}
Similarly to the implicit scheme (in positions) we can also make the masses and the linear operators implicit. The following result is a weak sphere barrier to external varifolds principle, in the sense that it holds for a small time step $\tau$ (depending on the distances between the points and the center of the sphere). Concretely:

\begin{prop}[Weak external varifold comparison principle]\label{prop:comparison_extvar_disc}

Assume that $(V^k)_k$ is a sequence of points clouds varifolds solution to the equation \eqref{def:disc_motion2} for $\Pi=2{ \rm Id}$ and that $V^0:=V$.

\noindent Let $z\in \R^n$, $\tau>0$; let $p \in  \{ 1, \dots, N \}$ be such that 
$$  |x^{k}_p-z|=\min\limits_{i\in  \{ 1, \dots, N\}} \abs{x^{k}_{i}-z}.$$ 
We have
$$\abs{x^{k}_{p}-z}^2 \leq \abs{x_{p}^{k+1}-z}^2 + 2d\tau.$$
Moreover, if we choose $\tau$  small enough such that 
\begin{equation}
 \exists p\in \lbrace 1, \ldots ,N \rbrace \,\,  \text{such that} \,\, |x^{k}_p-z|=\min\limits_{i\in  \{ 1, \dots, N\} }\abs{x^k_i-z}\,\,  \text{and} \,\,  |x^{k+1}_p-z|=\min\limits_{i\in  \{ 1, \dots, N\} }\abs{x^{k+1}_i-z};
\end{equation}
then, 
$$ \supp \|V^k\| \subset B(a,R)^c \,\, \implies \, \, \supp \|V^{k+1}\| \subset B(a,\sqrt{R^2-2d\tau})^c \, \, \forall (a,R) \in \R^n\times \R^+.$$
\end{prop}
\begin{proof}
Without loss of generality we assume that $z=0$. Fix $k\in \N$, let $\tau >0$ and $p$ be such that $ x^{k}_p=\min\limits_{i\in  \{ 1, \dots, N\}} \abs{x^{k}_{p}}$. For any $j\in \{ 1, \dots, N \}$ we have 
$$ (x_p^k-x_j^k) \cdot (x_p^k) = \abs{x_p^k}^2-x_j^k \cdot x_p^k \leq
\frac{1}{2}\abs{x_p^k}^2 -x_j^k \cdot x_p^k +\frac{1}{2}\abs{x_j^k}^2 = \frac{1}{2} \abs{x_p^k-x_j^k}^2.$$
\noindent So that,
\begin{equation*}
    \begin{split}
        \abs{x^{k}_{p}}^2  &= x^{k+1}_{p}\cdot x_{p}^{k} - \frac{\tau}{\e}\sum\limits_{j=1}^{N}\omega_{pj}^{k} 2\left( x_j^{k}-x_p^{k}\right) \cdot x_{p}^{k} \\&
        \leq \abs{x_{p}^k}\abs{x_{p}^{k+1}} + \frac{\tau}{\e}\sum\limits_{j=1}^{N}\omega_{pj}^{k} 2\left( x_p^{k}-x_j^{k}\right) \cdot x_{p}^{k} \\&
        \leq \abs{x_{p}^k}\abs{x_{p}^{k+1}} + \frac{\tau}{\e}\sum\limits_{j=1}^{N}\omega_{pj}^{k} \abs{x_p^k-x_j^k}^2 \\& 
        \leq \frac12 \abs{x_{p}^k}^2 + \frac12\abs{x_{p}^{k+1}}^2 + d\tau \quad \text{using \eqref{nkpcondition}}.
    \end{split}
\end{equation*}
therefore $\abs{x^{k}_{p}}^2 \leq \abs{x_{p}^{k+1}}^2 + 2d\tau$.\\
Now choose $\tau$  small enough such that 
\begin{equation}
 \exists p\in \lbrace 1, \ldots ,N \rbrace \,\,  \text{such that} \,\, |x^{k}_p|=\min\limits_{i\in  \{ 1, \dots, N\} }\abs{x^k_i}\,\,  \text{and} \,\,  |x^{k+1}_p|=\min\limits_{i\in  \{ 1, \dots, N\} }\abs{x^{k+1}_i}.
\end{equation}
If we assume  $\abs{x^{k}_{p}}>R$ then 
\begin{equation*}
 \min\limits_{i \in \{ 1, \dots, N \}} |x_i^{k+1}|^2 = |x_p^{k+1}|^2> R^2-2d\tau,
\end{equation*}
and this concludes the proof.
\end{proof}
%

\newpage
\section{Appendix}
For the sake of completeness we include in this appendix some of the results needed to prove the results of the current chapter.
\noindent We state a simple version of Brakke theorem on the tilt of tangent planes and the graph of varifolds.
\begin{theo}[Theorem 5.7, \cite{menne08}]
\label{thm:tilttex-heightex}
Let $V \in V_d(\R^n)$ be a integral varifold with $H(\cdot,V) \in \L^1(\V)$ and $(\delta V)_s \equiv 0$, assume that $0 \in \supp\|V\|$. Then
\begin{equation}
\int_{B_{\e}} \|T_yV-T_0V\|^2 \, d\|V\|(y) = o(\e^{d+1}), \quad \text{and} \,\, \int_{B_{\e}} |y-T_yV(y)|^2 \, d\|V\|(y) = o(\e^{d+3}).
\end{equation}

\end{theo}
The following theorem allows the use of parametrizations in Lemma \ref{lemma1_prop3.3br} and the proofs of Proposition \ref{propbr3.3_gen1} and \ref{propbr3.3_gen2}. We state a qualitative version, the original version can be found in \cite[Chapter 8]{all} 
\begin{theo}[Allard regularity theorem]\label{thm:allard_reg}
Let $V \in V_d(\R^n)$ be a integral unit density varifold with $H(\cdot,V)\in \cL^p(\V), p > d$ and $(\delta V)_s \equiv 0$, assume that $0$ is a generic point of  $ \supp\|V\|$. Then, near $0$, $\supp \|V\|$ is the graph of a map $ f \in C^{1,\alpha}(T_0V,\R^{n-d})$, $\alpha = 1-\frac{d}{p}$, $f(0) = 0_{\R^{n-d}}$ and $Df(0)= 0_{\R^{n-d}\times \R^d}$.
\end{theo}

\noindent We exhibit a key lemma on polynomial approximations that allows to prove Proposition \ref{propbr3.3_gen2}, it is a corollary of \cite[Proposition 4.1]{schatzle}.
\begin{lemma}[Polynomial approximation] \label{poly_approx}
Let $V$ be a $(n-1)$-integral unit density varifold in $\R^{n}$, such that $(\delta V)_s \equiv 0$, $H(\cdot,V)\in \cL^p(\V)$ with $p > n-1$; assume that $0$ is a generic point of  $ \supp\|V\|$. Then, there exists a polynomial map $q$ of degree at most $2$ such that : 
%
%
\begin{equation}
 \sup\limits_{\vert z \vert \leq \e } \vert q(z)- f(z) \vert = o(\e^2)
\end{equation}
where $z (\in T_0V) \mapsto (z,f(z))$ is a parametrization of $\supp \V$ near $0$ (which exists thanks to Allard's regularity).
\end{lemma}

\section*{Acknowledgments}
I would like to express my sincere gratitude to my thesis supervisors, Blanche Buet, Gian-Paolo Leonardi, and Simon Masnou, for their insightful suggestions and support throughout this work. I am also deeply thankful to Professor Abdelghani Zeghib, whose teaching made me fall in love with mathematics.

\newpage
\bibliographystyle{abbrv}
\bibliography{these}
\end{document}